\documentclass[12pt,a4paper]{article}
\usepackage[T1]{fontenc}

\usepackage{float}
\usepackage{tikz-cd}
\usepackage{amsmath}%
\usepackage{amsfonts}%
\usepackage{amsthm}
\usepackage{amssymb}%
\usepackage{graphicx}
\usepackage{hyperref}

\usepackage{geometry}
\usepackage{layout}
\usepackage[utf8]{inputenc}
\usepackage{amsmath}
\usepackage{graphicx}
\usepackage{amsmath}
\usepackage{amssymb}
\usepackage{amsfonts}
\usepackage{graphicx}
\usepackage{url}
\usepackage{MnSymbol}
\usepackage{arydshln}
\usepackage{outlines}
\usepackage{tikz-cd}

\usepackage{subfig}
\usepackage{amsthm}
\newcommand{\Fbar}{\overline{\mathbb{F}}}

\newcommand{\F}{\mathbb{F}}
\newcommand{\PGL}{\mathrm{PGL}}
\newcommand{\Z}{\mathbb{Z}}

\renewcommand{\P}{\mathbb{P}}

\newcommand{\Stab}{\mathrm{Stab}}
\newcommand{\Orbit}{\mathrm{Orbit}}

\newcommand{\Aut}{\mathrm{Aut}}

\newcommand{\Hyp}{\mathrm{Hyp}}
\newcommand{\Trig}{\mathrm{Trig}}

\newcommand{\ComInt}{\mathrm{ComInt}}

\usepackage{multirow}

\theoremstyle{plain}
\newtheorem{thm}{Theorem}[section]
\newtheorem{lem}[thm]{Lemma}

\theoremstyle{definition}

\newtheorem{exmp}{Example}[section]

\theoremstyle{remark}

\renewcommand{\>}{\right \rangle}
\usepackage{amsmath}

\usepackage{graphicx}

\title{Computing binary curves of genus five}
\author{Dušan Dragutinović
\\ Utrecht University\\
\textit{d.dragutinovic@uu.nl}}
\date{}

\begin{document}

\maketitle

\begin{abstract}
Genus 5 curves can be hyperelliptic, trigonal, or non-hyperelliptic non-trigonal, whose model is a complete intersection of three quadrics in $\P^4$. We present and explain algorithms we used to determine, up to isomorphism over $\F_2$, all genus 5 curves defined over $\F_2$, and we do that separately for each of the three mentioned types. We consider these curves in terms of isogeny classes over $\F_2$ of their Jacobians or their Newton polygons, and for each of the three types, we compute the number of curves over $\F_2$ weighted by the size of their $\F_2$-automorphism groups.   
\end{abstract}

\section{Introduction}
\label{Census_5}

A standard result 
is that the smooth curves of genus $5$ are either hyperelliptic, trigonal, or complete intersections of three quadric hypersurfaces in $\P^4$. Therefore, to understand the moduli space $\mathcal{M}_5$ of smooth curves of genus 5, we should understand the subvarieties parametrizing these three kinds of smooth curves. Denote with $\mathcal{H}_5$ the subvariety of $\mathcal{M}_5$ parametrizing hyperelliptic curves of genus 5, with $\mathcal{T}_5$ the subvariety parametrizing trigonal curves of genus 5, and lastly, let $\mathcal{U}_5$ be the subvariety parametrizing curves whose canonical model in $\P^4$ is a complete intersection of three quadric hypersurfaces. 
Let us write $\Hyp_g(\F_2)$ for the set of hyperelliptic curves of genus $g$ over $\F_2$ up to isomorphism \allowbreak (over $\F_2$),\allowbreak $\Trig_5(\F_2)$ for the set of trigonal curves of genus $5$ over $\F_2$ up to isomorphism (over $\F_2$), and $\ComInt_5(\F_2)$ for the set of curves of genus $5$ over $\F_2$ up to isomorphism (over $\F_2$) that are complete intersections of three quadric hypersurfaces in $\P^4$ in their canonical models. 

This paper aims to give algorithms for computing all $\F_2$-isomorphism classes of smooth curves of genus $5$ defined over $\F_2$, to present the obtained results, and to discuss some relevant questions, such as describing the isogeny types or Newton polygons of Jacobians of dimension $5$ over $\F_2$. We do that separately for $\Hyp_5(\F_2)$, $\Trig_5(\F_2)$, and $\ComInt_5(\F_2)$. Furthermore, we are interested in determining the automorphism groups over $\F_2$ for curves in the preceding three sets. These computations lead us to finding the moduli counts $|\mathcal{H}_5(\F_2)|, |\mathcal{T}_5(\F_2)|, |\mathcal{U}_5(\F_2)|$, and thus $|\mathcal{M}_5(\F_2)|$. Consequently, we get a piece of information about the cohomology of $\mathcal{M}_5$ and $\overline{\mathcal{M}}_5$.

In \cite{xarles}, Xarles determined all curves of genus $4$ defined over $\F_2$. His approach for computing the hyperelliptic curves is a universal one. We follow it closely and apply it to genus $5$ hyperelliptic curves. 

The algorithm for computing the representatives for the isomorphism classes of trigonal curves is based on the explicit description of their models in $\P^2$ and the idea of the exhaustion of all eligible equations respecting the isomorphisms; a stack count we obtained for this case matches the one by Wennink in \cite{wennink}, who used a partial sieve method for plane curves for that purpose. 

Lastly, for the non-hyperelliptic non-trigonal curves, a similar but more subtle idea of exhaustion of the eligible triples of quadratic polynomials in $\F_2[X, Y, Z, T, U]$ was used. We extensively explain the steps in our reasoning preceding the final algorithm we used for this problem and mention some intermediate steps and partial results.  
\\

\noindent
We implement all the algorithms and do the computations in the mathematical software \textsc{SageMath} and we collect the implementations and obtained results for non-hyperelliptic non-trigonal curves of genus 5 over $\F_2$ on 
\begin{center}
\url{https://github.com/DusanDragutinovic/Non-hyp-non-trig-genus-5-over-F_2}.
\end{center}
Codes and results for hyperelliptic and trigonal genus 5 curves over $\F_2$ can be found on 
\begin{center}
\url{https://github.com/DusanDragutinovic/MT_Curves}
\end{center}
and were made for the purpose of a master's thesis conducted at Utrecht University.

\subsection*{Acknowledgement}
I am grateful to my supervisor Carel Faber for pointing out to me the article by Xavier Xarles, which was a starting point for my master's thesis, all the discussions, and valuable help with completing the computations. 

Also, I would like to thank Lazar Mitrović and Miljan Zarubica for discussions regarding the technical aspects of the computations and for improving the execution time of the codes.
\newpage
\section{Hyperelliptic curves}
It is known that any hyperelliptic curve of genus $g$ over $\F_2$ can be represented in a standard (affine) equation
\begin{equation}
\label{eq:1}
 y^2 + q(x)y = p(x), \quad \text{for } p(x), q(x) \in \F_2[x], 
\end{equation}
with $2g + 1 \leq \max\{2 \deg(q(x)), \deg(p(x))\} \leq 2g + 2$.
\\

\noindent
In \cite{xarles}, Xarles gave the approach to compute all (smooth) curves of genus $4$ over $\F_2$ up to isomorphism. The given algorithm for determining the hyperelliptic curves over $\F_2$ can be generalized to higher genera, and here, we will use it to obtain the set $\Hyp_5(\F_2)$. Some of the claims made in \cite{xarles} we can use directly, while for the other ones, we will mention the analogs in the genus $5$ case. 
\\

\noindent Let $\F_2[x]_{n} = \{h(x)\in \F_2[x]: \deg(h(x))\leq n\}$ for $n \in \Z_{\geq 0}$, and for $A = \bigl(\begin{smallmatrix}
a & b\\ 
c & d
\end{smallmatrix}\bigr) \in \PGL_2(\F_2)$ and $q(x)\in \F_2[x]_{n}$, define an action of $\PGL_2(\F_2)$ on $\F_2[x]_{n}$ by $$\psi_n(A)(q(x)) = (cx + d)^{n}q\left (\frac{ax + b}{cx + d}\right );$$ we will also use the notation $A.q(x)$ for this. Further, denote the quotient set of $\F_2[x]_{n}$ under this action by $\overline{\F_2[x]_{n}} = \F_2[x]_{n}/\PGL_2(\F_2)$. \\

\noindent Let $H_1: y^2 + q_1(x)y = p_1(x)$ and $H_2: y^2 + q_2(x)y = p_2(x)$ be two hyperelliptic curves defined over $\F_2$, where it holds that $2g + 1 \leq \max\{2 \deg(q_i(x)), \deg(p_i(x))\} \leq 2g$ with $q_i(x)$ monic,  for $i \in \{1, 2\}$. Using that any isomorphism of such $H_1$ and $H_2$ has to be of the form $$(x, y)\mapsto \left (\frac{ax + b}{cx + d}, \frac{r(x) + y}{(cx + d)^{g + 1}}  \right )$$ for some $A = \bigl(\begin{smallmatrix}
a & b\\ 
c & d
\end{smallmatrix}\bigr) \in \PGL_2(\F_2)$ and $r(x)\in \F_2[x]_{g + 1}$, Xarles showed the following lemma. 

\begin{lem}[\cite{xarles}, Lemma 1]\label{lem:xarles_lem1}
Let $H_1$ and $H_2$ be as above and suppose $H_1\cong H_2$. Then there exists $A \in \PGL_2(\F_2)$ such that $q_2(x) = \psi_{g + 1}(A)(q_1(x))$. 
\end{lem}

\noindent
Using the lemma, when finding all non-isomorphic hyperelliptic curves $H: y^2  + q(x)y = p(x)$ of genus $5$ over $\F_2$, for a polynomial $q(x)$ it is enough to consider only elements of $\overline{\F_2[x]_{6}}$. 

 For any $q(x)\in \overline{\F_2[x]_{g + 1}}$, let $\Stab(q(x))$ be the stabilizer of $q(x)$ under the $\PGL_2(\F_2)$-action. We cite two more results from \cite{xarles}.

\begin{lem}[\cite{xarles}, Lemma 4]\label{lem:xarles_lem2}
Let $H_1$ and $H_2$ be two hyperelliptic curves of genus $g$ over $\F_2$ given by standard equations \eqref{eq:1}: $y^2 + q(x) = p_i(x)$, $i\in \{1,2\}$. If $H_1$ and $H_2$ are isomorphic over $\F_2$, then there are $A \in \mathrm{Stab}(q(x))$ and $r(x)\in \F_2[x], \deg(r(x))\leq g + 1$ such that $$p_2(x) = \psi_{2g + 1}(p_1(x) + r(x)^2 + q(x)r(x)).$$
\end{lem}

\begin{lem}[\cite{xarles}, Lemma 5]\label{lem:xarles_lem3}
Let $g \in \Z_{\geq 2}$. Given a nonzero polynomial $q(x)\in \F_2[x]$ and a polynomial $\allowbreak {p(x)\in \F_2[x] }\allowbreak$ with $2g + 1 \leq \max\{2\deg(q(x)), \deg(p(x))\}\leq 2g + 2$, the equation $\allowbreak {y^2 + q(x)y = p(x) }\allowbreak$ defines a hyperelliptic curve of genus $g$ if and only if $$\gcd(q(x), p'(x)^2 + q'(x)^2p(x)) = 1, $$ and either $\deg(q(x)) = g + 1$ or $a^2_{2g + 1} \neq a_{2g + 2}b_g^2$, where $p(x) = \sum_{i = 0}^{2g + 2}a_ix^i$ and $q(x) = \sum_{i = 0}^{g + 1}b_ix^i$. 
\end{lem}

\noindent The idea of determining $\Hyp_5(\F_2)$ is as follows. Initially, we should check for a pair of $p(x), q(x) \in \F_2[x]$  with $$11 \leq \max\{2 \deg(q(x)), \deg(p(x))\} \leq 12$$ whether $y^2 + q(x)y = p(x)$ defines a hyperelliptic curve of genus $5$ over $\F_2$. Lemma \ref{lem:xarles_lem1} reduces that, by considering some smaller set of possible $q(x)$'s, namely, only the set of representatives for $\F_2[x]_{6}$ for the $\PGL_2(\F_2)$-action, called $Q_5(\F_2)$. In other words, $Q_5(\F_2)$ is the set of representatives of elements in $\overline{\F_2[x]_{6}}$. Then, using Lemma \ref{lem:xarles_lem2}, for fixed $q(x)\in Q_5(\F_2)$, we can reduce the list of possible polynomials $p(x)$, and finally, Lemma \ref{lem:xarles_lem3} helps us to decide whether such pairs $(q(x), p(x))$ define hyperelliptic curves of genus $5$. Therefore, it is only left to determine $Q_5(\F_2)$. We do that below using the same ideas as in \cite{xarles}, Lemma 2. 

\begin{lem}\label{lem:xarles_help}
For $q(x)\in \F_2[x]_6$ and $\mathcal{Z}'(q(x)) = \{P \in \Fbar_2: q(P) = 0\}$, let $$D_{q(x)} = \mathcal{Z}'(q(x)) + (6 - \deg(q(x)))\cdot \infty$$ be the  zero divisor of $q(x)$ in $\P^1$. Then the action of $\PGL_2(\F_2)$ on $\F_2[x]_6$ naturally translates to the (standard) action of $\PGL_2(\F_2)$ on $\mathrm{Div}_{6}(\F_2)$, and these actions are compatible, i.e., $D_{A.q(x)} = A.D_{q(x)}$.
\end{lem}

\begin{proof}
For an arbitrary polynomial $q(x) = e_6x^6 + e_5x^5 + \ldots + e_1x + e_0 \in \F_2[x]_6$ and a matrix $A = \bigl(\begin{smallmatrix}
a & b\\ 
c & d
\end{smallmatrix}\bigr) \in \PGL_2(\F_2)$ we compute $$q_{new}(x) = A.q(x) = e_6(ax + b)^6 + e_5(ax + b)^5(cx + d) + \ldots + e_1(ax + b)(cx + d)^5 + e_0(cx + d)^6.$$ For $P = d/c, c\neq 0$, we see that $P \in \mathcal{Z}'(A.q(x))$ if and only if $\deg(q(x))<6$, and moreover, its multiplicity as a zero of $A.q(x)$ is precisely $6 - \deg(q(x))$; this means that the multiplicity of $P = d/c$ in $D_{A.q(x)}$ is the same as the multiplicity of $\infty$ in $D_{q(x)}$. Using $A^{-1}$ and changing the roles of $q(x)$ and $q_{new}(x)$ we can similarly get the conclusion on the degree of $q_{new}(x)$ when inspecting $P = \infty$. For other $P \in \Fbar_2$, we see $P \in \mathcal{Z}'(q(x))$ if and only if $\frac{aP + b}{cP + d} \in \mathcal{Z}'(q(x))$ and the corresponding multiplicities match. From these explicit relations, we see that the asserted claim holds. 
\end{proof}

\noindent
The previous lemma implies that determining $\overline{\F_2[x]_6}$ (and hence $Q_5(\F_2)$) is the same as determining $\text{Div}_{6}(\F_2)/\PGL_2(\F_2)$. Note further that if $P$ is a $K$-point, for $K/\F_2$ some finite extension, and $A = \bigl(\begin{smallmatrix}
a & b\\ 
c & d
\end{smallmatrix}\bigr) \in \PGL_2(\F_2)$, then $A.P = \frac{aP + b}{cP + d}$ is again a $K$-point ($\infty$ is a $\F_2$-point).

\begin{thm} The set $Q_5(\F_2)$ consists of the following elements $q(x)\in \F_2[x]$: 
\begin{itemize}
    \item[]$\deg(q(x))\leq 2$:  $\quad 1, x, x^2, x(x + 1), x^2 + x + 1$
    \item[]$\deg(q(x)) =  3$:  $\quad x^3, x^2(x + 1), (x^2 + x + 1)x, x^3 + x + 1$
    \item[]$\deg(q(x)) =  4$: $\quad x^2(x + 1)^2, (x^2 + x + 1)^2, (x^2 + x + 1)x^2, (x^2 + x + 1)x(x + 1), (x^3 + x + 1)x,\allowbreak (x^3 + x^2 + 1)x, x^4 + x + 1, x^4 + x^3 + 1$ 
    \item[]$\deg(q(x)) =  5$:  $\quad (x^2 + x + 1)^2x, (x^3 + x + 1)(x^2 + x + 1), (x^3 + x + 1)x(x+1), (x^4 + x + 1)x,\allowbreak (x^4 + x^3 + x^2 + x + 1)x, x^5 + x^2 + 1, x^5 + x^3 + 1, x^5 + x^3 + x^2 + x + 1$
    \item[]$\deg(q(x)) =  6$:  $\quad (x^2 + x + 1)^3, (x^3 + x + 1)^2, (x^3 + x + 1)(x^3 + x^2 + 1), (x^4 + x + 1)(x^2 + x + 1),\allowbreak x^6 + x + 1, x^6 + x^3 + 1$. 
\end{itemize}
\end{thm}
\begin{proof}
For $q(x)\in \F_2[x]_6$, let $D_{q(x)}$ be as in Lemma \ref{lem:xarles_help}. As we mentioned above, in order to find $Q_5(\F_2)$, we will firstly determine $\mathrm{Div}_{6}(\F_2)/\PGL_2(\F_2)$. We use the well-known fact that given any three $\F_2$-points $p_{\infty}, p_{0}, p_1$ there is a (unique) projective automorphism $A\in \PGL_2(\F_2)$ that sends $p_{\infty}\mapsto \infty, p_0\mapsto 0, p_{1}\mapsto 1$. 

Firstly, any $D_{q(x)}$ that consists only of $\F_2$-points in $\text{Div}_{6}(\F_2)/\PGL_2(\F_2)$ is equal to the unique one $n_{\infty}\cdot\infty + n_0\cdot 0 + n_1\cdot 1$ with $n_{1}\leq n_0 \leq n_{\infty}$. Since $\deg(D_{q(x)}) = 6$, we get that all the possible triples $(n_{\infty}, n_0, n_1)$ are $$\{(6, 0, 0), (5, 1, 0), (4, 2, 0), (4, 1, 1),(3, 3, 0), (3, 2, 1), (2, 2, 2)\}.$$ Using the correspondence from Lemma \ref{lem:xarles_help}, this gives us the subset of polynomials $q(x)$ in $Q_5(\F_2)$, $$\{1, x, x^2, x(x + 1), x^3, x^2(x + 1), x^2(x + 1)^2\}.$$ 

If $D_{q(x)}$ contains only one point of degree $2$ and no other points of degree $\geq 2$ in its support, similarly as above, we get that $D_{q(x)}$ is equal to one of $$3\zeta_2,\text{ } 2\zeta_2 + 2\infty,\text{ } 2\zeta_2 + \infty + 0,\text{ } \zeta_2 + 4\infty,\text{ } \zeta_2 + 3 \infty +  0,\text{ } \zeta_2 + 2\infty + 2\cdot 0,\text{ }  \zeta_2 + 2\infty + 0 +  1.$$ This induces the set of polynomials in $Q_5(\F_2)$ (we use $\zeta_n$ as notation for any $\zeta_n \in \Fbar_2$ of degree $n$ over $\F_2$): $$\{(x^2 + x + 1)^3, (x^2 + x + 1)^2, (x^2 + x + 1)^2x, x^2 + x + 1, (x^2 + x + 1)x, (x^2 + x + 1)x^2, (x^2 + x + 1)x(x + 1)\}, $$ 

If $D_{q(x)}$ contains a point of degree $3$, then the possibilities are the following $$D_{q(x)}\in \{2\zeta_3, \zeta_{3} + \zeta_{3}', \zeta_3 +  \zeta_2 + \infty, \zeta_3 + 3\infty, \zeta_3 + 2\infty + 0, \zeta_3 + \infty +  0 + 1\}, $$ where $\zeta_3, \zeta_3'$ are of degree $3$. Since $x\mapsto x + 1$, which is induced by the action of $A =  \bigl(\begin{smallmatrix}
1 & 1\\ 
0 & 1
\end{smallmatrix}\bigr)$, translates $q_1(x) = x^3 + x + 1$ to $q_2(x) = x^3 + x^2 + 1$, we have that $D_{q_1(x)}$ and $D_{q_2(x)}$ are the same in  $\text{Div}_{6}(\F_2)/\PGL_2(\F_2)$, and that $D_{(q_1(x))^2}, D_{(q_1(x))(x^2 + x + 1)}, D_{(q_1(x))x(x + 1)}$ are the same as $D_{(q_2(x))^2},\allowbreak D_{(q_2(x))(x^2 + x + 1)}, D_{(q_2(x))x(x + 1)}$. Therefore, this case gives us the new list of possible polynomials $q(x)$: $\{(x^3 + x + 1)^2, \allowbreak(x^3 + x + 1)(x^3 + x^2 + 1),\allowbreak (x^3 + x + 1)(x^2 + x + 1),\allowbreak x^3 + x + 1, (x^3 + x + 1)x,$ $(x^3 + x^2 + 1)x, (x^3 + x + 1)x(x + 1)\}.$

In the case when $D_{q(x)}$ contains a point of degree $4$, it should be either $\zeta_4 + \zeta_2, \zeta_4 + 2\infty$ or $\zeta_4 + \infty + 0$. There are three irreducible polynomials over $\F_2$ of degree $4$, so out of all possible combinations, discussing the $\PGL_2(\F_2)$ action on $\F_2[x]_6$ as above, we extract the following list of representatives for $q(x)$: $$\{(x^4 + x + 1)(x^2 + x + 1), x^4 + x + 1, x^4 + x^3 + 1, (x^4 + x + 1)x, (x^4 + x^3 + x^2 + x + 1)x\}.$$

When $D_{q(x)}$ contains a point of degree $5$, there is only one possibility for the form of $D_{q(x)}$, namely $D_{q(x)} = \zeta_5 + \infty$. Among the six irreducible polynomials of degree $5$, we found that for example, the following three are representatives of $q(x)$ for the considered action: $$\{x^5 + x + 1, x^5 + x^3 + 1, x^5 + x^3 + x^2 + x + 1\}.$$

Lastly, among the nine irreducible polynomials of degree $6$, we found two, namely $x^6 + x + 1, x^6 + x^3 + 1$, so that acting via $\PGL_2(\F_2)$ on them we can get all the others. This corresponds to a choice of the divisor $D_{q(x)} = \zeta_6$, with $\zeta_6$ (a point of degree $6$, which is either) a zero of $x^6 + x + 1$ or a zero of $x^6 + x^3 + 1$ in $\Fbar_2$.  
\end{proof}

\noindent
The previously described reasoning leads to an algorithm for computing the set $\Hyp_5(\F_2)$, that is practically the same as the one for computing $\Hyp_4(\F_2)$ from \cite{xarles}. 

\begin{center}
\textbf{Algorithm 1. Determine $\Hyp_5(\F_2)$.}  
\end{center}

\begin{itemize}
    \item[Step 0] From the previous theorem, we get $\mathbf{list\_of\_qs}$, the list of all possible representatives for a polynomial $q(x)$. 
    
    \item[Step 1] For each $q(x)$ in $\mathbf{list\_of\_qs}$ compute the stabilizer $\mathrm{Stab}(q(x)) \subseteq \PGL_2(\F_2)$ of $q(x)$ under the action defined by $\psi_{6}(\bigl(\begin{smallmatrix}
a & b\\ 
c & d
\end{smallmatrix}\bigr))(q(x)) = (cx + d)^6q(\frac{ax + b}{cx + d})$ for $\bigl(\begin{smallmatrix}
a & b\\ 
c & d
\end{smallmatrix}\bigr) \in \PGL_2(\F_2)$. 
    
    \item[Step 2] For fixed $q(x)$ in $\mathbf{list\_of\_qs}$, check whether a polynomial $p(x)\in \F_2[x]$, which satisfies $\allowbreak 11 \leq \max\{2 \deg(q(x)), \deg(p(x))\} \leq 12\allowbreak$, is such that $C: y^2 + q(x)y = p(x)$ is a (nonsingular) curve; collect all such $p(x)$'s in the list $\mathbf{q\_list\_of\_ps}$ of potential $p(x)$'s for $q(x)$. The smoothness condition can be checked using Lemma \ref{lem:xarles_lem3}, saying that $C$ is a (nonsingular) curve of genus $5$ if and only if $\gcd(q(x), p'(x)^2 + q'(x)^2p(x)) = 1$ and either $\deg(q(x)) = 6$ or $a^2_{11} \neq a_{12}b_{5}^2$, where $p(x) = \sum^{12}_{i=0} a_ix^i$ and $q(x) = \sum^{6}_{i=0}b_ix^i.$

    \item[Step 3] Fix $q(x)$ in $\mathbf{list\_of\_qs}$ and consider $\mathbf{q\_list\_of\_ps}$, its associated list of potential $p(x)$'s. For curves $C_1: y^2 + q(x)y = p_1(x)$ and $C_2: y^2 + q(x)y = p_2(x)$, we write $p_1(x) \sim p_2(x)$ if they are isomorphic over $\F_2$. Refine $\mathbf{q\_list\_of\_ps}$ by taking only the representatives $p(x)$ for this relation $\sim$. With the same argument as in Lemma \ref{lem:xarles_lem2}, we find that the relation $\sim$ is defined as: $p_1(x) \sim p_2(x)$ if and only if $(cx + d)^{12}p_2(\frac{ax + b}{cx + d}) = p_1(x) + r(x)^2 + r(x)q(x)$ for some $\bigl(\begin{smallmatrix}
a & b\\ 
c & d
\end{smallmatrix}\bigr) \in \mathrm{Stab}(q(x))$ and some $r(x)\in \F_2[x]$ of degree $\deg(q(x))\leq 6$. 
\end{itemize}

In such a manner, using the mathematical software \textsc{SageMath}, we computed the list of all non-isomorphic hyperelliptic curves of genus $5$ defined over $\F_2$. There are in total $1070$ such curves, i.e., $|\Hyp_5(\F_2)| = 1070$, and we confirmed that $$|\mathcal{H}_5(\F_2)| = \sum_{C \in \Hyp_5(\F_2)}\frac{1}{|\Aut_{\F_2}(C)|} =  512 = 2^{2\cdot5 - 1}.$$ For curves in $\Hyp_5(\F_2)$, we computed their numbers of points over finite fields $\F_{2^N}$, for $N \in \{1, 2, 3, 4, 5\}$, and then we found their Newton polygons. In particular, all the Newton polygons but one occur for elements of this class of genus $5$ (nonsingular) curves over $\F_2$. We mention explicitly some of the results in Section \ref{sec:results}.

\section{Trigonal curves}
Let $C$ be a (projective) plane curve. We say that a singularity of $C$ is of \textbf{delta invariant} $1$ if it is either a node (an ordinary double point), where a curve is locally of the form $xy = 0$, or an ordinary cusp, so that the curve is locally $y^2 = x^3$. 
\\
\\
A standard computation using the Riemann-Roch theorem and the genus-degree formula give us a well-known fact, which we use to compute $\Trig_5(\F_2)$:
\begin{thm}
A curve $C$ of genus $5$ is trigonal if and only if it can be represented as a plane quintic with precisely one singularity of delta invariant $1$.
\end{thm}
 
\noindent
 Any isomorphism of such curves $C_1$ and $C_2$ extends to an automorphism of $\P^2$. 
For a matrix $M = \begin{pmatrix}
m_{11} & m_{12} & m_{13}\\ 
m_{21} & m_{22} & m_{23}\\
m_{31} & m_{32} & m_{33}
\end{pmatrix}$ in $\PGL_3(\F_2)$ and $q(X, Y, Z)$ a homogeneous polynomial in $\F_2[X, Y, Z]$, the formula  
\begin{equation}
\label{eq:acttrig}
M.q(X, Y, Z) = q(m_{11}X + m_{12}Y + m_{13}Z, m_{21}X + m_{22}Y + m_{23}Z, m_{31}X + m_{32}Y + m_{33}Z)
\end{equation}
defines an action of $\PGL_3(\F_2)$ on the set of all homogeneous polynomials in $\F_2[X, Y, Z]$. Alternatively, we can define the mentioned action $M.q(X, Y, Z)$ as $q(M.(X, Y, Z))$ with $$M.(X, Y, Z) = M\cdot(X, Y, Z)^t.$$ 

Therefore, in order to determine the list of all trigonal curves of genus $5$ defined \allowbreak{over $\F_2$,} \allowbreak it is sufficient to find $\PGL_3(\F_2)$-representatives among all the quintic homogeneous polynomials in $X, Y, Z$ that define projective plane curves with delta invariant 1. 
 \\
 
 \noindent  To compute all trigonal curves of genus $5$ over $\F_2$, we have implemented the following algorithm in \textsc{SageMath}.
\\
\\

\allowbreak
\begin{center}
\textbf{Algorithm 2. Determine $\Trig_5(\F_2)$.}  
\end{center}
\begin{itemize}
    \item[Step 1] Make a list of all the monomials in $X, Y, Z$ of degree $5$ and fix the order of these, e.g. the lexicographic order \small{$X^5 > X^4Y > X^4Z >$ $X^3Y^2 > X^3YZ > X^3Z^2 > X^2Y^3 > X^2Y^2Z > X^2YZ^2 > X^2Z^3 > XY^4 > XY^3Z > XY^2Z^2 >$  $XYZ^3 > XZ^4 > Y^5 > Y^4Z > Y^3Z^2 > Y^2Z^3 > YZ^4 > Z^5$}. \normalsize Since there are $21$ monomials, we can represent all homogeneous polynomials of degree $5$ using coordinates of $\mathbb{P}^{20}({\F_2})$. Call the list of $21$-tuples $\mathbf{quintics}$. \\
    (\small{\textit{For example, $X^5 + YZ^4 \longleftrightarrow (1:0:0:0:0:0:0:0:0:0:0:0:0:0:0:0:0:0:0:1:0)$ under this correspondence}.})\normalsize
    \item[Step 2] From the starting list $\mathbf{quintics}$ obtain the new list $\mathbf{quintics\_repr}$ consisting only of representatives under the action of $\PGL_3(\F_2)$ on $X, Y, Z$. 
    \item[Step 3] Deduce whether a plane quintic corresponding to an element of $\mathbf{quintics\_repr}$ has exactly one singularity of order $2$ to reduce the previous list, and get the list $\mathbf{good\_quintics}$. 
    \item[Step 4] For all quintics with exactly one singularity $P$ of order $2$, represented by elements of  $\mathbf{good\_quintics}$, find a $\PGL_3(\F_2)$-isomorphic quintic with a singularity at $(0:0:1)$, such that its tangent space at $(0:0:1)$ is either $xy$ or $x^2 + xy + y^2$  (nodal case), or $y^2$ (potentially cuspidal case). The resulting list is $\mathbf{good\_quintics\_001}$.
    
    \item[Step 5] For all potentially cuspidal quintics decide whether there is a $\PGL_3(\F_2)$-isomorphic quintic with lowest terms $y^2 + x^3$, since only they are cuspidal with delta invariant $1$. Collect all such quintics, as well as the nodal quintics from $\mathbf{good\_quintics\_001}$ into the resulting list $\mathbf{trigonal\_curves}$.

\end{itemize}
\allowbreak
We found $2854$ trigonal curves in total that are not isomorphic (via $\PGL_3(\F_2)$-transformation), i.e. $|\Trig_5(\F_2)|=2854$, and we computed their automorphism groups over $\F_2$. In particular, we have obtained that $$|\mathcal{T}_5(\F_2)|:= \sum_{C \in \Trig_5(\F_2)}\frac{1}{|\mathrm{Aut}(C)|},$$ the number of (non-isomorphic) smooth trigonal curves of genus $5$ defined over the finite field with two elements weighted by the size of their automorphism group, i.e. the stack count for trigonal curves of genus 5 over $\F_2$, precisely equals $$|\mathcal{T}_5(\F_2)| = 2817 = 2^{11} + 2^{10} - 2^8 + 1.$$ This matches Wennink's results from \cite{wennink}, where he, using a partial sieve method for plane curves, computed these weighted numbers for any finite field with $q$ elements $\F_q$, and obtained $|\mathcal{T}_5(\F_q)| = q^{11} + q^{10} - q^8 + 1$. 

\newpage

\section{Complete intersections of three quadrics in $\P^4$}
The remaining curves of genus $5$ over $\F_2$ are the ones whose canonical embedding in $\mathbb{P}^4$ is a complete intersection of three quadric hypersurfaces. In other words, these curves $C$ are of the form $$C = Z(q_{P})\cap Z(q_{Q}) \cap Z(q_{R}) = Z(q_P, q_Q, q_R)  ,$$ where $q_P, q_Q, q_R \in \F_2[X, Y, Z, T, U]$ are homogeneous geometrically irreducible of degree $2$ with no non-trivial $\F_2$-linear combination among them. 

Denote the set of all the non-isomorphic representatives over $\F_2$ of these curves by $\ComInt_5(\F_2)$. These non-isomorphic curve classes, weighted by reciprocal of the size of their automorphism groups (over $\F_2$), which we denote $\mathcal{U}_5(\F_2)$, are the $\F_2$-points of the non-hyperelliptic non-trigonal locus $\mathcal{U}_5$ inside $\mathcal{M}_5$. 

The idea behind computing curves $C$ in $\ComInt_5(\F_2)$ is  as follows. 
First, we give a set $\Sigma$ of triples $(q_P, q_Q, q_R)$ of homogeneous quadratic polynomials in $\F_2[X, Y, Z, T, U]$ so that for any $C$ in $\ComInt_5(\F_2)$, there is an element of $\Sigma$ defining a curve isomorphic to $C$. Then, we filter the set $\Sigma$ to get the set $\ComInt_5(\F_2)$ so that no two elements of it define curves isomorphic over $\F_2$. Recall that the curves with canonical embedding into $\P^4$ are isomorphic over $\F_2$ if and only if their canonical models in $\P^4$ are isomorphic via some projective automorphism $M \in \PGL_5(\F_2)$.
\\

\noindent
 The group $\PGL_5(\F_2)$ acts on the subset of homogeneous quadratic polynomials in the polynomial ring $\F_2[X, Y, Z, T, U]$ by acting on variables $X, Y, Z, T, U$ via
\begin{equation}
M.(X, Y, Z, T, U) = M\cdot(X, Y, Z, T, U)^t
\label{eq:PGL_action_on_quadrics}    
\end{equation}
\noindent
for $M \in \PGL_5(\F_2)$. This induces the right action on homogeneous quadratic polynomials.

For practical work in the program \textsc{SageMath}, we represent the quadrics 
\begin{align*}
q_P = p_0 X^2 &+ p_1 XY + p_2 XZ + p_3 XT + p_4 XU + p_5Y^2 + p_6YZ + p_7YT + \\
&+ p_8YU +  p_9Z^2 + p_{10}ZT + p_{11}ZU + p_{12}T^2 + p_{13}TU + p_{14}U^2
\end{align*}
\noindent
using the $15$-tuples $P \in (\F_2)^{15} - \{0\}$ of the coefficients in $q_P$, $$P = (p_0, p_1, p_2, p_3 , p_4, p_5, p_6, p_7, p_8, p_9, p_{10}, p_{11}, p_{12}, p_{13}, p_{14}).$$ 

As usual, with $\Stab(q)$ and $\Orbit(q)$ we denote the stabilizer and the orbit of a homogeneous quadratic polynomial $q$ with respect to the action \eqref{eq:PGL_action_on_quadrics}. Using the correspondence above, for a $15$-tuple $P \in (\F_2)^{15} - \{0\}$ we set $$\Stab(P):= \Stab(q_P) \quad \text{and} \quad \Orbit(P):=\Orbit(q_P).$$ These notations follow by inducing the action of $\PGL_5(\F_2)$ on $(\F_2)^{15}-\{0\}$, where we define $M.P = Q$, for $M \in \PGL_5(\F_2)$ and $P, Q \in (\F_2)^{15}-\{0\}$ if it holds that $M.q_P = q_Q$. 
\\

\noindent First, we determine the set $\Sigma$.
\\

\noindent
\begin{center}
\textbf{Algorithm 3. Determine $\Sigma$.}  
\end{center}
\textit{\underline{Description.} Form a set $\Sigma$ of triples of quadrics $(q_P, q_Q, q_R)$ such that for any non-hyperelliptic non-trigonal curve $C$ of genus five over $\F_2$, there is a triple $(q'_P, q'_Q, q'_R)$ in $\Sigma$ such that $C$ is isomorphic to $Z(q'_P, q'_Q, q'_R)$ over $\F_2$.}

\begin{itemize}
    \item[Step 1]  Among all (nonzero) $15$-tuples representing the quadratic polynomials, find the representatives for the $\PGL_5(\F_2)$-action irreducible over $\Fbar_2$. There are seven representatives for the $\PGL_5(\F_2)$-action. Two of them are reducible over $\F_2$, while one of them is irreducible over $\F_2$, but not over $\F_4$. The remaining four are irreducible over $\Fbar_2$. 
    
    \textit{Output: $\mathbf{list\_of\_Ps}$; call its elements $P_1, P_2, P_3$, and $P_4$.} 
    
    (\textit{Using matrices from $\PGL_5(\F_2)$, we can always find an isomorphism of curves, $\allowbreak C = Z(q_1, q_2, q_3)\allowbreak \cong C'= Z(q_P, q_Q, q_R) \allowbreak$ over $\F_2$ with $q_P$ in $\mathbf{list\_of\_Ps}$.}) 
    
    \item[Step 2] For $P = P_4$, let $\textbf{P\_potential\_QRs}$ be the list consisting of all elements contained in the union $\Orbit(P_1)\cup \Orbit(P_2)\cup \Orbit(P_3)\cup \Orbit(P_4)$. For $P = P_3$, let $\textbf{P\_potential\_QRs}$ be $\Orbit(P_1)\cup \Orbit(P_2)\cup \Orbit(P_3)$. For $P = P_2$, let $\textbf{P\_potential\_QRs}$ be $\Orbit(P_1)\cup \Orbit(P_2)$. Lastly, let $\textbf{P\_potential\_QRs}$ be $\Orbit(P_1)$ for $P = P_1$.
    
    Fix $P$ in $\textbf{list\_of\_Ps}$, and find the representatives for the second quadric. Among the elements of $\mathbf{P\_potential\_QRs}$, find the representatives for the $\Stab(P)$-action. We do that by taking an element $Q$ from the list $\mathbf{P\_potential\_QRs}$ of $15$-tuples, putting it in $\textbf{P\_list\_of\_Qs}$, and removing from the $\mathbf{P\_potential\_QRs}$ all the $\Stab(P)$-conjugates of $q_Q$ and $(q_Q + q_P)$. 
    
    \textit{Output: $\textbf{P\_list\_of\_Qs}$ for $P = P_1, P_2, P_3, P_4$.}
    
    (\textit{As above, using matrices from $\Stab(q_P)$, we can find an isomorphism over $\F_2$ $$C = Z(q_1, q_2, q_3) \cong C' = Z(q_P, q_Q, q_R)$$ with $P$ in $\mathbf{list\_of\_Ps}$ and $Q$ in $\mathbf{P\_list\_of\_Qs}$.})    
    
    \item[Step 3] Find the representatives for the third quadric. Fix $P$ in $\mathbf{list\_of\_Ps}$ and fix $Q$ in $\mathbf{P\_list\_of\_Qs}$. Using the same reasoning as in Step 2, start from the whole list $\mathbf{P\_potential\_QRs}$ (depending only on $P$), take an element $R$ of it and move it to the list $\mathbf{PQ\_list\_of\_Rs\_apriori}$ and erase all the $(\Stab(P)\cap \Stab(Q))$-conjugates of $R, P+R, Q + R, P + Q + R$ from $\mathbf{P\_potential\_QRs}$. Repeat this until $\mathbf{P\_potential\_QRs}$ is empty. 
    
    For fixed $P$ and $Q$ as above and $R$ in $\mathbf{PQ\_list\_of\_Rs\_apriori}$, check whether all the elements from $(\F_2\cdot P + \F_2\cdot Q + \F_2\cdot R)-\{0\}$ are in $\mathbf{P\_potential\_QRs}$ (for such $P$), and whether the triple $(q_P, q_Q, q_R)$ defines a non-singular curve over $\Fbar_2$. 
\end{itemize}
\noindent
For each $P$ in $\mathbf{list\_of\_Ps}$, we ended up with a list of triples $(P, Q, R)$ defining non-singular curves. Call them $\mathbf{P\_list}$, for $P = P_1, P_2,  P_3, P_4$. The union of all these lists $\mathbf{P\_list}$ is the desired set $\Sigma$.

\begin{thm}
For any non-hyperelliptic non-trigonal curve $C$ of genus five over $\F_2$, there is an element $P$ in $\mathbf{list\_of\_Ps}$ and a triple $(P, Q, R)$ in $\mathbf{P\_list}$ so that $C \cong Z(q_P, q_Q, q_R)$ over $\F_2$.
\end{thm}
 
\begin{proof}
Take $C = Z(q_1, q_2, q_3)$. If in $(\F_2\cdot q_1 + \F_2\cdot q_2 + \F_2\cdot q_3)-\{0\}$ there is an element $q$ belonging to the $\Orbit(q_{P_4})$, we may without of loss of generality assume that $q_1 = q$ and set $q_2 = q_2', q_3 = q_3'$ so that $\<q_1, q_2, q_3\> = \<q, q_2', q_3'\>$. Write $P = P_4$. Thus, we can isomorphically map $C$ to $C' = Z(q_P, q_R', q_Q')$ using a matrix $M \in \PGL_5(\F_2)$, such that $M.q_1 = q_P$, $M.q_2 = q_Q'$ and $M.q_3 = q_R'$. Then, using the  $\Stab(P)$-action we can isomorphically map $C'$ to a curve $C'' = Z(q_P, q_Q, q_R'') = Z(q_P, q_Q + q_P, q_R'')$ so that $q_Q'$ goes to either $q_Q$ or $q_Q + q_P$ for some $Q$ in $\mathbf{P\_list\_of\_Qs}$. Lastly, using the $\Stab(P)\cap \Stab(Q)$-action, we can isomorphically map $C''$ to $C''' = Z(q_P, q_Q, q_R)$ so that $q_R'''$ goes to an element from $\{q_R, q_R + q_P, q_R + q_Q, q_R + q_P + q_Q\}$ for $R$ in $\mathbf{PQ\_list\_of\_Rs}$. In particular, $$C \cong Z(q_P, q_Q, q_R),$$ for $(P, Q, R)$ in $\mathbf{P4\_list}$. 

Otherwise, all elements from $S = (\F_2\cdot q_1 + \F_2\cdot q_2 + \F_2\cdot q_3)-\{0\}$ are of the form $q_P$ for some $P$ in $\Orbit(P_1)\cup \Orbit(P_2)\cup \Orbit(P_3)$. If there is $q = q_P \in S$ so that $P$ belongs to $\Orbit(P_3)$, set $q_1 = q$, and proceed as above. If there is no such $q$, but there is $q = q_P$, with $P$ in $\Orbit(P_2)$, set $q_1 = q$ and repeat the argument. The remaining case is that all the elements from $(\F_2\cdot q_1 + \F_2\cdot q_2 + \F_2\cdot q_3)-\{0\}$ are of the form $q_P$ for some $P$ in $\Orbit(P_1)$; the argument goes analogously.

Therefore, our construction of lists $\mathbf{P\_list}$, for $P = P_1, P_2, P_3, P_4$ is satisfactory. Namely, by the above, any curve $C$ will have a(t least one) representative occurring in the union of lists $\mathbf{P\_list}$, for $P = P_1, P_2, P_3, P_4$.

Checking the condition of whether the curves are smooth over a field $\Fbar_2$ was done by using the function $\mathrm{is\_smooth()}$,  Furthermore, checking whether the considered varieties are curves indeed, i.e. one-dimensional, was done using the function $\mathrm{dimension()}$.
Both mentioned functions were already implemented in \textsc{SageMath}.
\end{proof}

\noindent
The second part of computing the isomorphism classes of non-hyperelliptic non-trigonal curves of genus five over the field with two elements consists of reducing the obtained lists $\mathbf{P\_list}$ for $P = P_1, P_2, P_3, P_4$, and obtaining the final list $\mathbf{Com\_Int\_5\_F2}$. For each curve $C \in \ComInt(\F_2)$, we want to have precisely one element of $\mathbf{Com\_Int\_5\_F2}$ representing the $\F_2$-isomorphism class of $C$.

Firstly, note that by our construction, for $i, j \in \{1, 2, 3, 4\}$ and $i\neq j$, no two triples $$(P, Q, R) \in \mathbf{Pi\_list}, \quad \text{and } (P', Q', R') \in \mathbf{Pj\_list}$$ can define isomorphic curves. Indeed, for any element $(P, Q, R) \in \mathbf{P4\_list}$ in the set $(\F_2\cdot q_P + \F_2\cdot q_Q + \F_2\cdot q_R)-\{0\}$ there is at least one element $q$ which is in $\Orbit(q_{P_4})$, which is not true for the elements in $\mathbf{Pi\_list}$ for $i = 1, 2, 3$. Also, for any element $(P, Q, R) \in \mathbf{P3\_list}$ in $(\F_2\cdot q_P + \F_2\cdot q_Q + \F_2\cdot q_R)-\{0\}$ there is at least one element $q$ which is in $\Orbit(q_{P_3})$, and there is none in $\Orbit(q_{P_4})$, and similarly for the remaining ones. 

Then, for a fixed list $\mathbf{Pi\_list}$, with $P_i$ one of $P_1, P_2, P_3$, or $P_4$, we take an element $(P, Q, R)$ of it, consider the vector space $(\F_2\cdot P + \F_2\cdot Q + \F_2\cdot R)-\{0\}$ and in it, we find all the elements $S$ in $\Orbit(P_i)$. The only possibility of these triples in $\mathbf{Pi\_list}$ to define the same curve is that some of these elements $S$ go to the $P_i$. Therefore, the idea is to map $S$ to $P_i$ using all possible $\PGL_5(\F_2)$-transformations to check whether the isomorphism is established. In this way, we will finally determine the lists of the desired representatives of the isomorphism classes. Note that two triples $(P, Q, R)$ and $(P', Q', R')$ in $\Sigma$ define the same curve if and only if $\<P, Q, R\>_{\F_2} = \<P', Q', R'\>_{\F_2}$.  

\begin{center}
\textbf{Algorithm 4. Determine $\ComInt_5(\F_2)$.}  
\end{center}
\textit{\underline{Description.} Consider the triples in $\Sigma$ one after the other. For a chosen triple $(P, Q, R)$, which defines a curve $C$, find all the triples in $\Sigma - \{(P, Q, R)\}$ defining a curve that is isomorphic over $\F_2$ to $C$ and remove them from $\Sigma$.}

\begin{itemize}
    \item[Step 0] Let $\mathbf{Com\_Int\_5\_F2}$ be an empty list at the beginning.
    
    \item[Step 1] For fixed $P = P_i$ in $\{P_1, P_2, P_3, P_4\}$ take a triple $(P, Q, R)$ in $\mathbf{P\_list}$. Add it to the final list $\mathbf{Com\_Int\_5\_F2}$. Let $C = Z(q_P, q_Q, q_R)$.
    \item[Step 2] For a triple $(P, Q, R)$ from Step 1, consider the vector space $W = \F_2\cdot P + \F_2\cdot Q + \F_2\cdot R$ and let $D = W\cap \Orbit(P_i)$. For each element in $D$, find all matrices $M \in \PGL_5(\F_2)$ mapping it to $P = P_i$. 
    \item[Step 3] Use the matrices $M$, obtained in Step 2, to act on $(P, Q, R)$. If $(M.P, M.Q, M.R)$ defines the same curve as some triple from $\mathbf{P\_list}$, remove such a triple from $\mathbf{P\_list}$.
\end{itemize}

\noindent Let us make a few remarks about the practical background of the implementation of \textit{Algorithm 4}.

To find all matrices $M \in \PGL_5(\F_2)$ mapping an element $S \in \Orbit(P_i)$ to $P_i$ it is enough to have only one such matrix $M_0$ and to know all matrices in $\Stab(P_i)$. Namely, it is not hard to see that for any such $S$ and $M$, there is a matrix $N \in \Stab(P_i)$ so that $M.S = N.(M_0.S)$.  

Note that in Step 3, we will always remove at least one element from $\mathbf{P\_list}$, namely $(P, Q, R)$, so that the process terminates. As discussed in the last paragraph before the presentation of \textit{Algorithm 4}, in this way, we will indeed remove from $\mathbf{P\_list}$ all the triples $(P', Q', R')$ defining a curve isomorphic over $\F_2$ to $C = Z(q_P, q_Q, q_R)$. Lastly, in Step 3, we check whether $(M.P, M.Q, M.R)$ defines the same curve as some triple from $\mathbf{P\_list}$ using the criterion occurring in the paragraph just mentioned, by checking the equality of the corresponding vector spaces. 
\\

\noindent
The obtained set $\ComInt_5(\F_2)$ consists of $3905$ elements. 

\subsection{Computing the automorphisms over $\F_2$}

An automorphism $C$ over $\F_2$, for $C \in \ComInt_5(\F_2)$ is induced by a matrix in $\PGL_5(\F_2)$. If $M \in \PGL_5(\F_2)$ is such that $M: C \overset{\cong}{\to} C$ and $C = Z(q_P, q_Q, q_R)$, then $$\<Mq_P, Mq_Q, Mq_R\>_{\F_2} = \<q_P, q_Q, q_R\>_{\F_2}.$$ Therefore, for triples $(P, Q, R)$ in $\mathbf{Com\_Int\_5\_F2}$ with $P = P_i \in \{P_1, P_2, P_3, P_4\}$, either $M \in \Stab(P)$ or $M$ maps an element of $V = (\F_2P + \F_2Q + \F_2R) - \{0, P\}$ to $P$. In the latter case, we see that $M$ needs to belong to the set of matrices mapping elements from $V \cap \Orbit(P)$ to $P$; call that set $D = D_{(P, Q, R)}$. Therefore, for fixed $(P, Q, R)$ as above, we can only check whether $$\F_2\cdot M.P + \F_2\cdot M.Q + \F_2\cdot M.R = \F_2P + \F_2Q + \F_2R$$ for $M$ in $\Stab(P)\cup D_{(P, Q, R)}$. 

From the previous discussion, we can easily get the precise steps of an algorithm for computing $\mathrm{Aut}_{\F_2}(C)$ for each $C \in \ComInt_5(\F_2)$. 
\\

\noindent  We implemented the algorithms from this section in \textsc{SageMath} and computed the sets $\Sigma$ and $\ComInt_5(\F_2)$, as well as $\mathrm{Aut}_{\F_2}(C)$, the automorphism groups over $\F_2$, for each curve $\allowbreak C \in \ComInt_5(\F_2). \allowbreak$ Moreover, for each $C \in \ComInt_5(\F_2)$, we computed the number of $\F_{2^{N}}$-points of $C$ for $N = 1, 2, 3, 4, 5$ to get the isogeny classes over $\F_2$, as well as to compute Newton polygons of the obtained curves.

\section{Obtained results}
\label{sec:results}
As we already indicated, using the algorithms from Section 3 for hyperelliptic, Section 4 for trigonal, and Section 5 for non-hyperelliptic non-trigonal curves of genus $5$ over $\F_2$, we computed $\Hyp_5(\F_2), \Trig_5(\F_2)$, and $\ComInt_5(\F_2)$, the sets of all the isomorphism representatives. For all the obtained curves, we computed the number of points over $\F_{2^N}$ for $N \in \{1, 2, 3, 4, 5\}$ and determined the elements of their automorphism groups. 

There are $1070$ hyperelliptic, $2854$ trigonal, and $3905$ non-hyperelliptic non-trigonal curves of genus 5 over $\F_2$, so in total, there are $7829$ pairwise non-isomorphic curves of genus 5 over $\F_2$.

We already mentioned the stack counts $$|\mathcal{H}_5(\F_2)| = 512 \quad \text{and} \quad |\mathcal{T}_5(\F_2)| = 2817,$$
and we got the stack count $$|\mathcal{U}_5(\F_2)| = \sum_{C \in \ComInt_5(\F_2)}\frac{1}{| \mathrm{Aut}_{\F_2}(C)|} = 3584 = 2^{12} - 2^9.$$ Therefore, we have $$|\mathcal{M}_5(\F_2)| = 6913 = 2^{12} + 2^{11} + 2^{10} - 2^8 + 1.$$

The Honda-Tate theorem gives us that there are $4339$ isogeny classes over $\F_2$ among Jacobians of genus $5$ curves defined over $\F_2$. 

Furthermore, Jacobian varieties of dimension $g = 5$ over $\F_2$ realize all eligible Newton polygons of height $2g = 10$. In other words, for any eligible Newton polygon $N$ of height $10$, there is a curve of genus $5$ defined over $\F_2$, which has $N$ as its Newton polygon. In Table \ref{table:genus5slopes}, for each of the three discussed classes of genus 5 curves, we mention the number of such curves occurring for indicated Newton polygons of height $10$.

\begin{table}[H]
\centering
\begin{tabular}{ |p{4.4cm}||p{2cm}|p{2cm}|p{2.4cm}|p{1cm}|  }
 \hline
 Newton polygon slopes&  $\Hyp_5(\F_2)$ & $\Trig_5(\F_2)$ & $\ComInt_5(\F_2)$ & Total\\
 \hline
 \small{$\left [0, 0, 0, 0, 0, 1, 1, 1, 1, 1\right ]$}  & 550 & 1417 &  1617 & 3584\\\hline
\small{$\left [0, 0, 0, 0, \frac{1}{2}, \frac{1}{2}, 1, 1, 1, 1\right ]$}&  156& 623& 868 & 1647 \\\hline
 \small{$\left [0, 0, 0, \frac{1}{2}, \frac{1}{2}, \frac{1}{2}, \frac{1}{2}, 1, 1, 1\right ]$} & 108& 404 & 672&1184\\\hline
 \small{$\left [0, 0, \frac{1}{3}, \frac{1}{3}, \frac{1}{3}, \frac{2}{3}, \frac{2}{3}, \frac{2}{3}, 1, 1\right ]$}& 32 & 122& 206& 360\\\hline
 \small{$\left [0, 0, \frac{1}{2}, \frac{1}{2}, \frac{1}{2}, \frac{1}{2}, \frac{1}{2}, \frac{1}{2}, 1, 1\right ]$}   & 88 & 80& 176 & 344\\\hline
 \small{$\left [0, \frac{1}{4}, \frac{1}{4}, \frac{1}{4}, \frac{1}{4}, \frac{3}{4}, \frac{3}{4}, \frac{3}{4}, \frac{3}{4}, 1\right ]$}&  0 & 64 & 88& 152\\\hline
 \small{$\left [0, \frac{1}{3}, \frac{1}{3}, \frac{1}{3}, \frac{1}{2}, \frac{1}{2}, \frac{2}{3}, \frac{2}{3}, \frac{2}{3}, 1\right ]$}&  48 &24& 40 & 112\\\hline
 \small{$\left [0, \frac{1}{2}, \frac{1}{2}, \frac{1}{2}, \frac{1}{2}, \frac{1}{2}, \frac{1}{2}, \frac{1}{2}, \frac{1}{2}, 1\right ]$}&  56 & 28 & 108& 192 \\\hline
 \small{$\left [\frac{1}{5}, \frac{1}{5}, \frac{1}{5}, \frac{1}{5}, \frac{1}{5}, \frac{4}{5}, \frac{4}{5}, \frac{4}{5}, \frac{4}{5}, \frac{4}{5}\right ]$}& 0 & 48 & 48 &  96\\\hline
\small{ $\left [\frac{1}{4}, \frac{1}{4}, \frac{1}{4}, \frac{1}{4}, \frac{1}{2}, \frac{1}{2}, \frac{3}{4}, \frac{3}{4}, \frac{3}{4}, \frac{3}{4}\right ]$}& 0&  8 & 24 &  32\\\hline
 \small{$\left [\frac{1}{3}, \frac{1}{3}, \frac{1}{3}, \frac{1}{2}, \frac{1}{2}, \frac{1}{2}, \frac{1}{2}, \frac{2}{3}, \frac{2}{3}, \frac{2}{3}\right ]$}&  16 & 18 & 26 &  60\\\hline
 \small{$\left [\frac{2}{5}, \frac{2}{5}, \frac{2}{5}, \frac{2}{5}, \frac{2}{5}, \frac{3}{5}, \frac{3}{5}, \frac{3}{5}, \frac{3}{5}, \frac{3}{5}\right ]$}& 8  & 4 & 4 & 16\\\hline
 \small{$\left [\frac{1}{2}, \frac{1}{2}, \frac{1}{2}, \frac{1}{2}, \frac{1}{2}, \frac{1}{2}, \frac{1}{2}, \frac{1}{2}, \frac{1}{2}, \frac{1}{2}\right ]$} & 8 & 14 & 28 & 50\\\hline
 \hline
 \end{tabular}
\caption{Numbers of curves for indicated Newton polygon}
\label{table:genus5slopes}
\end{table}

In Table \ref{table:genus5stackslopes}, we collect the stack counts for all three types of genus 5 curves over $\F_2$ possessing specified Newton polygon. 

\begin{table}[H]
\centering
\begin{tabular}{ |p{4.4cm}||p{2cm}|p{2cm}|p{2cm}|p{1.6cm}|  }
 \hline
 Newton polygon slopes&  $\mathcal{H}_5(\F_2)$ & $\mathcal{T}_5(\F_2)$ & $\mathcal{U}_5(\F_2)$ & $\mathcal{M}_5(\F_2)$\\
 \hline
 \small{$\left [0, 0, 0, 0, 0, 1, 1, 1, 1, 1\right ]$}  & 264 & 1405 &  1524 & 3193\\\hline
\small{$\left [0, 0, 0, 0, \frac{1}{2}, \frac{1}{2}, 1, 1, 1, 1\right ]$}&  76& 610& 838 & 1524 \\\hline
 \small{$\left [0, 0, 0, \frac{1}{2}, \frac{1}{2}, \frac{1}{2}, \frac{1}{2}, 1, 1, 1\right ]$} & 52& 402 & 574&1028\\\hline
 \small{$\left [0, 0, \frac{1}{3}, \frac{1}{3}, \frac{1}{3}, \frac{2}{3}, \frac{2}{3}, \frac{2}{3}, 1, 1\right ]$}& 16 & 122& 198& 336\\\hline
 \small{$\left [0, 0, \frac{1}{2}, \frac{1}{2}, \frac{1}{2}, \frac{1}{2}, \frac{1}{2}, \frac{1}{2}, 1, 1\right ]$}   & 40 & 78& 154 & 272\\\hline
 \small{$\left [0, \frac{1}{4}, \frac{1}{4}, \frac{1}{4}, \frac{1}{4}, \frac{3}{4}, \frac{3}{4}, \frac{3}{4}, \frac{3}{4}, 1\right ]$}&  0 & 64 & 88& 152\\\hline
 \small{$\left [0, \frac{1}{3}, \frac{1}{3}, \frac{1}{3}, \frac{1}{2}, \frac{1}{2}, \frac{2}{3}, \frac{2}{3}, \frac{2}{3}, 1\right ]$}&  24 &24& 32 & 80\\\hline
 \small{$\left [0, \frac{1}{2}, \frac{1}{2}, \frac{1}{2}, \frac{1}{2}, \frac{1}{2}, \frac{1}{2}, \frac{1}{2}, \frac{1}{2}, 1\right ]$}&  24 & 24 & 64& 112 \\\hline
 \small{$\left [\frac{1}{5}, \frac{1}{5}, \frac{1}{5}, \frac{1}{5}, \frac{1}{5}, \frac{4}{5}, \frac{4}{5}, \frac{4}{5}, \frac{4}{5}, \frac{4}{5}\right ]$}& 0 & 48 & 48 &  96\\\hline
\small{ $\left [\frac{1}{4}, \frac{1}{4}, \frac{1}{4}, \frac{1}{4}, \frac{1}{2}, \frac{1}{2}, \frac{3}{4}, \frac{3}{4}, \frac{3}{4}, \frac{3}{4}\right ]$}& 0&  8 & 24 &  32\\\hline
 \small{$\left [\frac{1}{3}, \frac{1}{3}, \frac{1}{3}, \frac{1}{2}, \frac{1}{2}, \frac{1}{2}, \frac{1}{2}, \frac{2}{3}, \frac{2}{3}, \frac{2}{3}\right ]$}&  8 & 14 & 18 &  40\\\hline
 \small{$\left [\frac{2}{5}, \frac{2}{5}, \frac{2}{5}, \frac{2}{5}, \frac{2}{5}, \frac{3}{5}, \frac{3}{5}, \frac{3}{5}, \frac{3}{5}, \frac{3}{5}\right ]$}& 4  & 4 & 4 & 12\\\hline
 \small{$\left [\frac{1}{2}, \frac{1}{2}, \frac{1}{2}, \frac{1}{2}, \frac{1}{2}, \frac{1}{2}, \frac{1}{2}, \frac{1}{2}, \frac{1}{2}, \frac{1}{2}\right ]$} & 4 & 14 & 18 & 36\\\hline
 \hline
 \end{tabular}
\caption{Stack counts for indicated Newton polygon}
\label{table:genus5stackslopes}
\end{table}

\begin{exmp}
The non-hyperelliptic non-trigonal curve $C_{3103}^{\mathrm{ComInt}}$ given in $\P^4$ by 
\begin{align*}
 C_{3103}^{\mathrm{ComInt}}: \begin{cases}
Y^2 + XZ + YZ = 0\\
XY + XZ + YT + ZT + XU + ZU + U^2 = 0\\
XY + XZ + YZ + Z^2 + XT + ZT + T^2 + YU + ZU = 0
\end{cases}
\end{align*}
is a unique curve with $9$ $\F_2$-points, which is the maximal number of $\F_2$-points among all genus $5$ curves over $\F_2$. The maximum number of $\F_2$-points among hyperelliptic curves (of genus $5$ over $\F_2$) is $6$, and there are $44$ such curves with this property. Among trigonal curves (of genus $5$ over $\F_2$), the maximum number of $\F_2$-points is $8$, and there are $6$ such trigonal curves. This agrees with a result by Faber and Grantham \cite{fabergrantham}.

Furthermore, there are $308$ curves of genus 5 over $\F_2$ without any $\F_2$-point; $44$ of them are hyperelliptic, $23$ of them are trigonal, and the remaining $241$ are in $\ComInt_5(\F_2)$.
\end{exmp}

\begin{exmp}
There are $161$ isogeny classes over $\F_2$ so that in each of them, we can find (pairwise non-isomorphic) Jacobians defined by all three types of genus $5$ curves. For example, the curves $C^{\mathrm{Hyp}}_{2}$ (hyperelliptic), $C^{\mathrm{Trig}}_{1}$ (trigonal), and $C^{\mathrm{ComInt}}_{124}$ (non-hyperelliptic non-trigonal) \allowbreak \text{given by} \allowbreak

$$C^{\mathrm{Hyp}}_{2}: y^2 + y + x^{11} + x^{10} + x^8 + x^7 + x^6 + x^5 + x^4 + x^3 + x^2 + x = 0,$$

\begin{align*}
C^{\mathrm{Trig}}_{1}:  X^4Y +& X^3Y^2 + XY^4 + X^3YZ + X^2Y^2Z + XY^3Z + \\ +& X^3Z^2 + X^2YZ^2 + Y^3Z^2 + XYZ^3 + Y^2Z^3 = 0 \text{ in } \P^2, \text{ and}
\end{align*}
$$C^{\mathrm{ComInt}}_{124}: \begin{cases}
Y^2 + YZ + Z^2 + XT + ZT = 0\\
XT + XU + YU + ZU = 0\\
X^2 + XY + Y^2 + XZ + YZ + XU + YU + ZU + TU + U^2 = 0
\end{cases} \text{ in } \P^4,$$
define Jacobians isogenous over $\F_2$. That is because for each curve $C$ of them we have that $|C(\F_2)| = 5, |C(\F_{2^2})| = 9, |C(\F_{2^3})| =  11, |C(\F_{2^4})| =  33,$ and $|C(\F_{2^5})| =  25$.
\end{exmp}

\begin{exmp}

In Table \ref{table:heauto}, we mention the sizes of $\F_2$-automorphism groups occurring for hyperelliptic curves of genus $5$ over $\F_2$ and the numbers of such curves $C$ with the indicated size $|\Aut_{\F_2}(C)|$.

\begin{table}[H]
\begin{center}
\begin{tabular}{ |c||c|c|c|c|  }
 \hline
 $|\Aut_{\F_2}(C)|$  & 2 & 4 &  6 & 12\\\hline
$\# C \text{ in } \Hyp_5(\F_2)$ &  983 & 76& 7 & 4 \\\hline
 \end{tabular}
\caption{Sizes of $\F_2$-automorphism groups of curves in $\Hyp_5(\F_2)$}
\label{table:heauto}
\end{center}
\vspace{-6mm}
\end{table}
\noindent
The following curves
\begin{align*}
C^{\mathrm{Hyp}}_{144}: &  y^2 + (x + 1)^2x^2y +  x^{11} + x^9 + x^8 + x^5 + x^3 + x^2  + x + 1 = 0\\
C^{\mathrm{Hyp}}_{145}: &  y^2 + (x + 1)^2x^2y + x^{11} + x^{10} +  x^3 +  x = 0\\
C^{\mathrm{Hyp}}_{457}: &  y^2  + (x^3 + x^2 + 1)(x^3 + x + 1)y + x^{10} + x^6 + x^4 + x^3  = 0\\
C^{\mathrm{Hyp}}_{458}: &  y^2 + (x^3 + x^2 + 1)(x^3 + x + 1)y + x^{12} + x^{10} + x^8 + x^7 + x^5 + x^3 + 1 = 0
\end{align*}
 with indices $144, 145, 457, 458$ in our list $\Hyp_5(\F_2)$ are the ones with $\F_2$-automorphism group of size $12$, the maximal one.
\end{exmp}

\begin{exmp}
As one can see in Table \ref{table:genus5slopes}, there are $8$ supersingular hyperelliptic curves of genus $5$ over $\F_2$. Their indices in our list $\Hyp_5(\F_2)$ are $0, 5, 9, 11, 15, 20, 25,$ and $31$. They all have $\F_2$-automorphism group of size $2$.
\end{exmp}

\begin{exmp} Let $\F_q$ be a finite field of cardinality $q = p^r$ for $r\in \Z_{> 0}$ and a prime number $p\in \Z_{>0}$, let $g\in \Z_{\geq 2}$, and let $\lambda = [1^{\lambda_1}, \ldots, \nu^{\lambda_{\nu}}]$ be a partition of an integer $m\in \Z_{\geq 1}$ with non-negative integers $\lambda_1, \ldots, \lambda_{\nu}$, such that $|\lambda|:=\sum_{i = 1}^{\nu}i\lambda_i = m$. In \cite{jonas}, Bergstr\"om made $\mathbb{S}_n$-equivariant counts of points defined over finite fields of the moduli space $\mathcal{H}_{g, n}$ of $n$-pointed smooth hyperelliptic curves of genus $g$ over $\F_q$. The counts depend on the numbers $$a_{\lambda}|_g:= \sum_{C \in \Hyp_g(\F_q)}\frac{1}{|\Aut_{\F_q}(C)|}\prod_{i = 1}^{\nu}a_i(C)^{\lambda_i}, $$ where $a_i(C) = q^i + 1 -|C(\F_{q^i})|$, for $i \in \Z_{\geq 1}$. In that paper, for arbitrary $q$ and $g$ as above, some explicit formulas were mentioned: $a_{[2]}|_g = (-1)^g - q^{2g}$, $a_{[1^2]}|_g = -1 + q^{2g},$ $$a_{[1^2, 2]}|_g = -\frac{q^{2g + 2} - 1}{q + 1} - q^{2g} + \frac{1}{2}g(q^3 + q - 2) + \frac{1}{2}\begin{cases}
2q & \text{ if } g\equiv 0 \mod 2 \\ 
q^3 - q - 2 & \text{ if } g\equiv 1 \mod 2 
\end{cases}, $$ and $a_{\lambda}|_g = 0$, if $|\lambda|$ is odd. Using our data, for $q = 2$ and $g = 5$, we computed the sums from the definition of $a_{\lambda}|_g$ for some $\lambda$, and the outcomes agree with the mentioned formulas: $$a_{[2]}|_5 = -1025, \quad a_{[1^2]}|_5 = 1023, \quad a_{[1^2, 2]}|_5 = -2367, \quad a_{[1^2, 3]}|_5 = a_{[3, 4^2]}|_5 = a_{[1, 2, 5^2]}|_5 = 0.$$

\end{exmp}

\begin{exmp}

In Table \ref{table:trigauto}, we mention the sizes of $\F_2$-automorphism groups occurring for trigonal curves of genus $5$ over $\F_2$ and the numbers of such curves $C$ with the indicated size $|\Aut_{\F_2}(C)|$.
\begin{table}[H]
\begin{center}
\begin{tabular}{ |c||c|c|c|c|  }
 \hline
 $|\Aut_{\F_2}(C)|$  & 1 & 2 &  3 & 6\\\hline
$\# C \text{ in } \Trig_5(\F_2)$ &  2783 & 63& 7 & 1 \\\hline
 \end{tabular}
\caption{Sizes of $\F_2$-automorphism groups of curves in $\Trig_5(\F_2)$}
\label{table:trigauto}
\end{center}
\vspace{-6mm}
\end{table}
\noindent
The curve $C^{\mathrm{Trig}}_{980}$ given by the equation in $\P^2$
\begin{align*}
C^{\mathrm{Trig}}_{980}:  X^5 +& Y^5 + X^4Z + X^3YZ + XY^3Z + Y^4Z + X^3Z^2 +\\
+& X^2YZ^2 + XY^2Z^2 + Y^3Z^2 + X^2Z^3 + XYZ^3 + Y^2Z^3 = 0,
\end{align*}
which has index $980$ in our list $\Trig_5(\F_2)$, is the one with $|\Aut_{\F_2}(C^{\mathrm{Trig}}_{980})| = 6$. 
\end{exmp}

\begin{exmp}
All $14$ supersingular trigonal curves have a trivial $\F_2$-automorphism group. Their indices in our list of trigonal genus 5 curves over $\F_2$ are $103, 612, 1044, 1059, 1144, \allowbreak 1150, \allowbreak 1292, 1384, 1667, 1684, 1728, 2387, 2460,$ and $2502$. 
\end{exmp}

\begin{exmp}
In Table \ref{table:nonhenontrigauto}, we mention the sizes of $\F_2$-automorphism groups occurring for non-hyperelliptic non-trigonal curves of genus $5$ over $\F_2$ and the numbers of such curves $C$ with the indicated size $|\Aut_{\F_2}(C)|$.

\begin{table}[H]
\begin{center}
\begin{tabular}{ |c||c|c|c|c|c|c|c|c|c|  }
 \hline
 $|\Aut_{\F_2}(C)|$  & 1 & 2 &  3 & 4& 6& 8& 12& 16& 24\\\hline
$\# C \text{ in } \ComInt_5(\F_2)$ &  3319 & 490 & 3 & 60 & 4& 24& 2& 2& 1 \\\hline
 \end{tabular}
\caption{Sizes of $\F_2$-automorphism groups of curves in $\ComInt_5(\F_2)$}
\label{table:nonhenontrigauto}
\end{center}
\vspace{-6mm}
\end{table}
\noindent
The curve with index $3491$ in our list of non-hyperelliptic non-trigonal curves, \allowbreak{given in } $\P^4$ by\allowbreak $$ C^{\mathrm{ComInt}}_{3491}: \begin{cases}
Y^2 + XZ + YZ = 0\\
Y^2 + XZ + YZ + Z^2 + YT + T^2 + XU + YU + ZU = 0\\
X^2 + Y^2 + XZ + YT + T^2 + YU + U^2 = 0  
\end{cases}$$
is the (only) one with $\F_2$-automorphism group of size 24; this is the largest possible size among all $\F_2$-automorphism groups of genus 5 curves defined over $\F_2$.
\end{exmp}

\begin{exmp}
The indices of the $28$ non-hyperelliptic non-trigonal supersingular curves of genus 5 over $\F_2$ are $4, 369, 375, 409, 487, 556, 1373, 1381, 1548, 2587, 2588, 2608, 2625, 2811,\allowbreak 2812, 2817, 2820, 2920, 2928, 2935, 3091, 3093, 3094, 3096, 3098, 3099, 3103$, and  $3756$. Four of them have automorphism group of size $8$, six of them of size $4$, and four of them of size $2$, while the others have a trivial automorphism group.  
\end{exmp}

\begin{exmp}
In the first step of Algorithm 3 for determining the set $\ComInt_5(\F_2)$ of non-hyperelliptic non-trigonal curves of genus 5 over $\F_2$, we mentioned that we found models for curves $C$ in $\ComInt_5(\F_2)$, so that one quadric on which such a curve lies is always one of the quadrics corresponding to the $15$-tuples $P_1, P_2, P_3$, or $P_4$. \allowbreak \text{Explicitly, we have} \allowbreak
\begin{align*}
q_{P_1} &=  X^2 + XZ + YZ + XT + ZT + TU, \\
q_{P_2} &= XY + Y^2 + Z^2 + YT + ZT, \\
q_{P_3} &=  Y^2 + YZ + Z^2 + XT + ZT, \\
q_{P_4} &= Y^2 + XZ + YZ. 
\end{align*}
In the union $\bigcup_{i \in \{1, 2, 3, 4\}} \Orbit(q_{P_i})$  there are $32116$ elements, so in the set difference between all quadratic polynomials and this union, there are $651$ elements. It can be checked that all of them are reducible over $\Fbar_2$ and that they split into three orbits as was mentioned in the description of Algorithm 3. Elements of the mentioned orbits, as well as other relevant data we used can be found in codes available on the mentioned \textsc{GitHub} page.

For completeness, let us mention the remaining three (reducible over $\Fbar_2$) quadratic polynomials we obtained in the first step of Algorithm 3. They are 
\begin{align*}
q_{P_5} &= Y^2, \\
q_{P_6} &= X^2 + XY, \\
q_{P_7} &= X^2 + XY + XT + Y^2 + YZ + YT + Z^2 + ZT + T^2 \\
&=(X+\zeta_2 Y +Z+(\zeta_2+1)T)(X+ (\zeta_2+1)Y +Z+\zeta_2T),
\end{align*} where $\zeta_2 \in \Fbar_2$, so that $\zeta_2^2 + \zeta_2 + 1 = 0$. 
\end{exmp}

\end{document}